\documentclass[11pt]{article}
\usepackage{currfile}
\usepackage{color}
\usepackage{mathrsfs}
\usepackage{graphicx}
\usepackage{caption}
\usepackage{subfigure}
\usepackage{amsmath}
\usepackage{amssymb}
\usepackage{setspace}

\newtheorem{theorem}{Theorem}[section]
\newtheorem{corollary}[theorem]{Corollary}

\newtheorem{definition}[theorem]{Definition}
\newtheorem{conjecture}[theorem]{Conjecture}

\newtheorem{lemma}[theorem]{Lemma}
\newtheorem{proposition}[theorem]{Proposition}
\newtheorem{claim}{Claim}

\newenvironment{proof}{\noindent {\bf
Proof.}}{\rule{3mm}{3mm}\par\medskip}

\newcommand{\ENDproof}{~\rule{3mm}{3mm}\medskip\par}

\newcommand{\JCT}{{\it J. Combin. Theory}, }

\newcommand{\ComHung}{{\it Combinatorica}, }

\newcommand{\CJM}{{\it Canad. J. Math.}, }

\begin{document}

\title{On $3$-flow-critical graphs}

\author{
 Jiaao Li$^1$, Yulai Ma$^2$\thanks{Corresponding author.}, Yongtang Shi$^2$, Weifan Wang$^3$, Yezhou Wu$^4$\\
\small $^1$School of Mathematical Sciences and LPMC, Nankai University, Tianjin 300071, China\\
\small $^2$Center for Combinatorics and LPMC, Nankai University,  Tianjin 300071, China\\
\small $^3$Department of Mathematics, Zhejiang Normal University, Jinhua, Zhejiang 321004, China\\
\small $^4$Ocean College, Zhejiang University, Zhoushan, Zhejiang 316021, China\\
\small Emails: lijiaao@nankai.edu.cn; ylma92@163.com; shi@nankai.edu.cn;\\
\small  wwf@zjnu.cn; yezhouwu@zju.edu.cn
}

\date{}

\maketitle

   \begin{abstract}
   A  bridgeless graph $G$ is called {\em $3$-flow-critical} if it does not admit a nowhere-zero $3$-flow, but $G/e$ has  for any $e\in E(G)$. Tutte's $3$-flow conjecture can be equivalently stated as that every $3$-flow-critical graph contains a vertex of degree three. In this paper, we study the structure and extreme  edge density of $3$-flow-critical graphs.   We apply structure properties to obtain lower and upper bounds on the density of  $3$-flow-critical  graphs, that is, for any $3$-flow-critical graph $G$ on $n$ vertices,
$$\frac{8n-2}{5}\le |E(G)|\le 4n-10,$$
where each equality holds if and only if $G$ is $K_4$. We conjecture that every $3$-flow-critical graph on  $n\ge 7$ vertices has at most $3n-8$ edges, which would be tight if true. For planar graphs, the best possible density upper bound of $3$-flow-critical graphs on $n$ vertices is $\frac{5n-8}{2}$, known from a result of Kostochka and Yancey (JCTB 2014) on vertex coloring $4$-critical graphs by duality.
\\[2mm]
   \end{abstract}
\section{Introduction}

Graphs in this paper are finite and may contain parallel edges but no loops. We follow \cite{BoMu08, Zhang1997} for undefined notation and terminology. A vertex of degree $k$ in a graph $G$ is called a $k$-vertex. Denote by $V_{k}(G)$ ($V_{\le k}(G)$ and $V_{\ge k}(G)$, respectively)  the set of all vertices of degree $k$ (at most $k$ and  at least $k$, respectively) in $G$.

Let $D=D(G)$ be an orientation of $G$. For a vertex pair $(u,v)$, denote  by $u\rightarrow v$ if there is an arc leaving $u$ and entering $v$.
For each $v\in V(G)$, we use $E^+_D(v)$ and $E^-_D(v)$ to denote the set of all arcs
directed out from $v$ and directed into $v$, respectively. An ordered pair $(D, f)$ is called an {\em integer flow} of $G$ if $D$ is an orientation and $f$ is a mapping from $E(G)$ to integers  such that every vertex $v\in V(G)$ is balanced, that is
$\sum_{e \in E^+_D(v)} f(e) - \sum_{e \in E^-_D(v)} f(e) =0.$
An integer flow $(D, f)$ is called a  {\em nowhere-zero $k$-flow} if $1\le|f(e)|\le k-1$,   $\forall e\in E(G)$.


As observed by Tutte \cite{Tutte54, Tutte66}, flow and coloring are dual concepts: a plane graph $G$ admits a nowhere-zero $k$-flow if and only if the dual graph $G^*$ is $k$-colorable. A graph $G$ is called vertex coloring {\em $4$-critical} if $G$ is not  $3$-colorable  but deleting any edge in $G$ results a $3$-colorable graph. Motivated by this, we define a bridgeless graph $G$ to be {\em $3$-flow-critical} if $G$ admits no nowhere-zero $3$-flow but $G/e$ has a nowhere-zero $3$-flow for each edge $e\in E(G)$.

The study of vertex coloring  $4$-critical graphs can be traced back to Dirac, Gallai and Ore in 1950s and 1960s (see \cite{KY2014}).  It follows from  Tur\'{a}n's Theorem that 
every $4$-critical graph on $n\ge 5$ vertices has at most $\frac{1}{3}n^2$ edges, since any such  graphs contain no $K_4$ as a subgraph.  In  \cite{Toft1970}, Toft constructed  $4$-critical graphs with more than $\frac{1}{16}n^2$ edges, while the optimal value remains unknown as of today. For the lower bound,   resolving conjectures of Gallai and Ore on the density of $4$-critical graphs, Kostochka and Yancey \cite{KY2014, KY2014ca} proved a tight bound that every $4$-critical graph on $n$ vertices has at least $\frac{5n-2}{3}$ edges. By duality, their theorem shows the following result on $3$-flow-critical planar graphs.

\begin{theorem}\label{planarKY} {\em (Kostochka and Yancey \cite{KY2014, KY2014ca} )}
  For any $3$-flow-critical planar graph $G$ on $n$ vertices, $$|E(G)|\le \frac{5}{2}n-4.$$
  Moreover, the equality holds if and only if $G$ is the dual of a planar $4$-Ore graph.
\end{theorem}

A natural question is to ask what is the corresponding density bound for nonplanar graphs. It is easy to see that the upper bound $\frac{5}{2}n-4$  for planar graphs does not hold for general graphs. One may verify that (see Prop \ref{k3c}) the graph $K_{3, n-3}^+$ (where $n\ge 6$) in Figure \ref{fig1} is $3$-flow-critical with $3n-8$ edges, where $K_{3, n-3}^+$ denotes the graph obtained from complete bipartite graph $K_{3,n-3}$ by adding a new edge between two vertices of degree $n-3$.


 In this paper, we provide  linear lower and upper bounds for all $3$-flow-critical graphs.

\begin{theorem}\label{THM: main4n}
  Let $G$ be a $3$-flow-critical graph on $n$ vertices. Then $$\frac{8n-2}{5} \le |E(G)|\le 4n-10,$$
  and each equality holds if and only if $G\cong K_4$. Moreover, we have $\frac{8n+2}{5} \le |E(G)|\le 4n-11$ if $G\not\cong K_4$.
\end{theorem}

\begin{figure}[ht]
        \centering
        \includegraphics[width=0.5\textwidth]{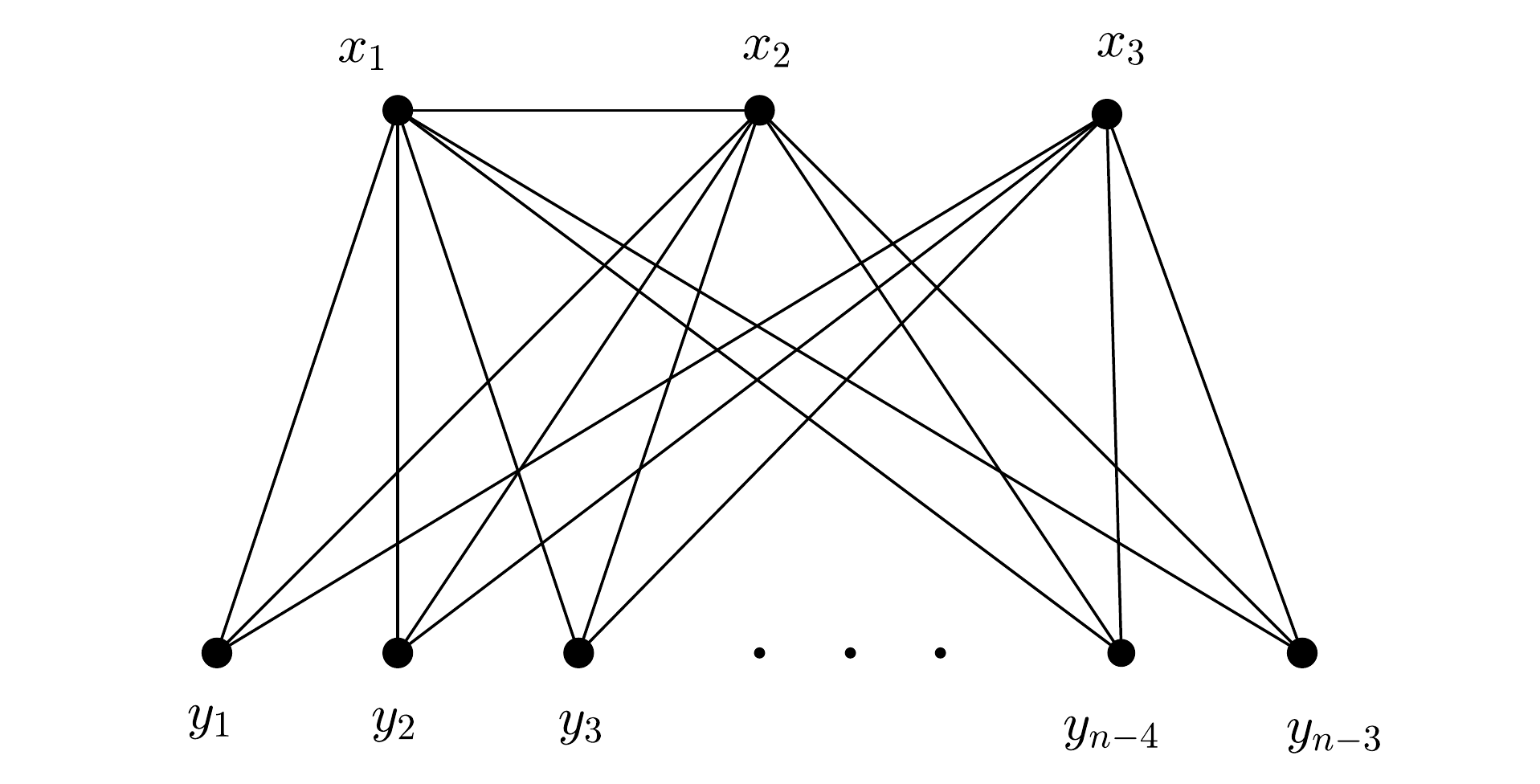}
          \caption{The graph $K_{3, n-3}^+$.}\label{fig1}
        \end{figure}
        
It seems that the bounds in Theorem \ref{THM: main4n} are not optimal in general. A construction of Yao and Zhou \cite{YZ17}  shows that there exists $3$-flow-critical planar graphs on $n$ vertices with $\frac{7n-1}{4} $ edges (see Theorem \ref{YZthm} below). Applying this result, we are able to construct $3$-flow-critical graphs with density from $\frac{7}{4}$ up to $3$ by developing a $2$-sum operation in Section 4. It seems unclear about the exact density lower bound.
But we conjecture that the tight upper bound should be $3n-8$ for $n\geq7$.

\begin{conjecture}\label{CONJ: 3n}
  For any $3$-flow-critical graph $G$ on $n\ge 7$ vertices, $$|E(G)|\le 3n-8.$$
\end{conjecture}

Perhaps  $K_{3, n-3}^+$ is the only extreme graph to attain this bound when $n$ is large. At least, it is true if $|V_3(G)|\geq n-3$, as shown in Prop \ref{up3n-8} in  Section 2.

Tutte's $3$-flow conjecture (see Unsolved Problems \#97 in \cite{BoMu08}) asserts  that every $4$-edge-connected graph admits a nowhere-zero $3$-flow. The density argument, even if Conjecture \ref{CONJ: 3n} was proved, cannot derive the $3$-flow conjecture. We propose a stronger conjecture below, which, if true, implies the $3$-flow conjecture.
\begin{conjecture}\label{CONJ: 2.5n}
  For any $3$-flow-critical graph $G$ on $n$ vertices, $$|E(G)|< \frac{5}{2}n +n_3,$$
  where $n_3=|V_3(G)|$ is the number of $3$-vertices in $G$.
\end{conjecture}
Note that $K_{3, n-3}^+$ satisfies Conjecture \ref{CONJ: 2.5n} since it has many $3$-vertices. There is another family of $3$-flow-critical graphs on $n$ vertices constructed from $2$-sum of $K_4$'s, which contains four $3$-vertices and $n-4$ $5$-vertices, approaching the bound in Conjecture \ref{CONJ: 2.5n}. To support Conjecture \ref{CONJ: 2.5n}, we provide the following result.
\begin{theorem}\label{THM: 9n8}
  For any $3$-flow-critical graph $G$ on $n$ vertices, $$|E(G)|< \frac{5}{2}n +9n_8,$$
  where $n_8=|V_{\le 8}(G)|$ is the number of  vertices of degree at most $8$ in $G$.
\end{theorem}

~

The rest of the paper is organized as follows. In section 2, we introduce a few basic notation and terminology, and then investigate structures of $3$-flow-critical graphs to prove the lower bound in Theorem \ref{THM: main4n}. In section 3, we complete the proof of the upper bound in Theorem \ref{THM: main4n} as well as the proof of Theorem \ref{THM: 9n8}. Finally, we  develop some operations to construct $3$-flow-critical graphs with density between $\frac{7}{4}$ and $3$ in section 4.

\section{Properties of $3$-flow-critical graphs}

For vertex subsets $U,  W\subseteq V(G)$, let
$[U, W]_G = \{uw \in E(G)|u \in U, w \in W\}$. When $U =  \{u\}$ or
$W =\{w\}$, we use $[u,W]_G$ or $[U,w]_G$ for $[U,W]_G$, respectively. The subgraph of $G$ induced by $U$ is denoted by $G[U]$. For any subset $S\subseteq V(G)$, we denote $S^c=V(G)-S$.  An  edge cut $[S,S^c]_G$ is called {\em essential} if there are at least two nontrivial components in $G-[S,S^c]_G$. A graph is called essentially $k$-edge-connected if it contains no essential edge cut with less than $k$ edges. When there is no scope for ambiguity, the subscript $G$ may be omitted. 
Contracting an edge of a graph means to identify  its two endpoints and then delete the resulting loops.
For an edge $e\in E(G)$ and a subgraph $H$ of $G$, we write $G/e$ to denote the graph obtained from $G$ by contracting $e$, 
and denote by $G/H$ as a graph obtained from $G$ by successively contracting the edges of $E(H)$.

Let $d^+_D(v)=|E^+_D(v)|$ and $d^-_D(v)=|E^-_D(v)|$ denote
the out-degree and the in-degree of $v$ under the orientation $D$,
respectively. Let $Z_n$ be the set of integers modulo $n$. A function $\beta$: $V(G) \rightarrow Z_3$ is a $Z_3$-boundary if $\sum_{v\in V(G)}\beta (v)\equiv 0\pmod3$. For a given $Z_3$-boundary $\beta$, a {\em $\beta$-orientation} is an orientation $D$ of $G$ such that $d^+_D(v) - d^-_D(v)\equiv \beta (v)\pmod3$ for each $v\in V(G)$. Especially, a  {\em modulo $3$-orientation} of $G$ is a $\beta$-orientation with $\beta (v) \equiv0\pmod3$ for each $v\in V(G)$. We call a graph $G$ {\em $Z_3$-connected}   if for any $Z_3$-boundary $\beta$ of  $G$, there exists a $\beta$-orientation of $G$.
A graph is called {\em $Z_3$-reduced} if it does not have any nontrivial $Z_3$-connected subgraphs.
It is well-known that a graph admits a nowhere-zero $3$-flow if and only if it admits a modulo $3$-orientation (see \cite{Younger1983, Zhang1997}). Therefore, in the rest of this paper we will study nowhere-zero $3$-flows in terms of modulo $3$-orientations.

A useful method to prove $Z_3$-connectedness  is the following lemma.
\begin{lemma}\label{lem1}{\em (Lai \cite{LaiH00})}
Let G be a graph, and let $H\subseteq G$ be a subgraph of $G$.

 (i) If $H$ is $Z_3$-connected  and $G/H$ has a modulo $3$-orientation, then $G$ has a modulo $3$-orientation.

 (ii) If both $H$ and $G/H$ are $Z_3$-connected, then $G$ is also $Z_3$-connected.

 (iii) The graph $2K_2$ is $Z_3$-connected, where $2K_2$ consists of two vertices and two parallel edges.
\end{lemma}

 A wheel graph  $W_k$ is  constructed by adding a new center vertex  connecting  to each vertex of a $k$-cycle, where $k\ge 3$. A wheel $W_k$ is odd if $k$ is  odd, and even otherwise.
 \begin{lemma}\label{lem-wheel} {\em(Fan, Lai, Xu, Zhang, Zhou \cite{FanZ_3})}
   A wheel $W_k$ is   $Z_3$-connected if and only if $k$ is even. Furthermore, each odd wheel does not admit a nowhere-zero $3$-flow.
 \end{lemma}
As an example, it is an easy exercise to verify that each odd wheel is  $3$-flow-critical  by Lemmas \ref{lem1} and \ref{lem-wheel}.

Our first result of this section is  the following fundamental structure properties of $3$-flow-critical graphs.
\begin{theorem}\label{THM: prop3f}
  Let $G$ be a $3$-flow-critical graph. Then each of the following holds.

  (i) For any $e\in E(G)$, $G-e$ admits a nowhere-zero $3$-flow.

  (ii) $G$ is $3$-edge-connected and essentially $4$-edge-connected.

  (iii) $G$ is $Z_3$-reduced.

  (iv) $G[V_3]$ contains no cycle, unless $G$ is an odd wheel.
\end{theorem}

\begin{proof}
(i)  For any $e=uv\in E(G)$, let $D$ be a modulo $3$-orientation of $G/e$.  Let $D^*$ be the restriction of $D$ on $G-e$. By  arbitrarily orienting each edge in $E(G-e) \setminus E(G/e)$ (if any), we obtain an orientation $D'$ of $G-e$.
If $D'$ is not a modulo $3$-orientation of $G-e$, then  either $d_{D'}^+(u)-d_{D'}^-(u)\equiv d_{D'}^-(v)-d_{D'}^+(v)\equiv 1 \pmod3$ or $d_{D'}^+(u)-d_{D'}^-(u)\equiv d_{D'}^-(v)-d_{D'}^+(v)\equiv -1 \pmod3$. So $D'$ can be extended to a modulo $3$-orientation of $G$ by letting $v \rightarrow u$ or $u \rightarrow v$, a contradiction. Hence, $G-e$ has a modulo $3$-orientation, and so it admits a nowhere-zero $3$-flow.


(ii) Let  $[S,S^c]_G$ be an  edge cut of $G$. 
Suppose $|[S,S^c]_G|=2$. Denote $[S,S^c]_G=\{u_1v_1,u_2v_2\}$ and $\{u_1,u_2\}\subseteq S$. By definition, $G/u_1v_1$ has a modulo $3$-orientation, and, WLOG, assume $u_2\rightarrow v_2$. Transfer this orientation to $G$ and let $v_1\rightarrow u_1$ to obtain an orientation $D$ of $G$. We have that every vertex, except perhaps $u_1$ and $v_1$, is balanced in $D$.
Then we know that every vertex except perhaps $v_1$ is balanced under the restriction of $D$ on $G/G[S]$, and so is $v_1$ by looking the sum of difference between outdegree and indegree. Thus, the orientation $D$ is  balanced at $v_1$, and thus at $u_1$. This completes a modulo $3$-orientation of $G$, a contradiction.

%
 Now suppose $|[S,S^c]_G|=3$, $|E(G[S])|\geq1$ and  $|E(G[S^c])|\geq1$. Assume $e_1\in E(G[S])$ and $e_2\in E(G[S^c])$. By definition, $G/e_1$ admits a modulo $3$-orientation $D'$ such that the three edges in $[S,S^c]_G$ are all leaving or entering $S^c$. Then the restriction $D_1$ of $D'$ on $G/G[S]$ is a modulo $3$-orientation. By symmetry, $G/G[S^c]$ has a modulo $3$-orientation $D_2$. Then either $D_1$ and $D_2$ agree along $[S,S^c]_G$ directly, or they agree after reversing all edge directions in $D_2$. Thus, their union provides a modulo $3$-orientation of $G$, a contradiction. Hence $G$ is $3$-edge-connected and essentially $4$-edge-connected.

 (iii) Suppose that $H$ is a nontrivial $Z_3$-connected subgraph of $G$. Let $u_1v_1\in E(H)$. By (i), $G-u_1v_1$ admits a modulo $3$-orientation $D_1$. Thus the restriction  $D'$ of $D_1$ on $G/H$ is also a modulo $3$-orientation. By Lemma \ref{lem1}, $G$ has a modulo $3$-orientation, a contradiction. So $G$ is $Z_3$-reduced.


 (iv) Suppose, by contradiction, that $G$ is not an odd wheel and  $G[V_3]$ contains a cycle. Assume $C=v_1v_2\ldots v_tv_1$ is the minimal cycle
in $G[V_3]$. Note that $C$ is an induced subgraph of $G$.  Let $u_i$ be the neighbor of $v_i$ which is not on $C$ and let $e_i=u_iv_i$. 

First, suppose $t$ is even. By (i), $G-e_1$ admits a modulo $3$-orientation $D'$. It implies that $d^+_{D'}(v_i)=3$ or $d^-_{D'}(v_i)=3$  for each $i\in\{2,3,\ldots,t\}$. Since $t$ is even,  we have  $d^+_{D'}(v_2)=d^+_{D'}(v_t)=3$ or $d^-_{D'}(v_2)=d^-_{D'}(v_t)=3$, which implies  that $d^-_{D'}(v_1)=2$ or $d^+_{D'}(v_1)=2$. So $v_1$ is not balanced in $D'$. This leads to a contradiction.

Next, suppose $t$ is odd. If there exists an edge $e$ that is not incident to any vertex on $C$, then by (i), $G-e$ admits a modulo $3$-orientation $D'$. It implies that $d^+_{D'}(v_i)=3$ or $d^-_{D'}(v_i)=3$  for each $i\in\{1,2,\ldots,t\}$.  Since $t$ is odd,  we have either  $d^+_{D'}(v_2)=d^-_{D'}(v_t)=3$ or $d^-_{D'}(v_2)=d^+_{D'}(v_t)=3$, which implies that $v_1$ is not balanced in $D'$, a contradiction.
Hence we suppose $E(G)=E(C)\cup\{e_1, e_2, \ldots, e_t\}$. 
Since $G$ is not an odd wheel, there exists an index $j\in\{1,2,\ldots,t\}$ such that $u_j\neq u_{j+1}$.
By (i), $G-e_j$ admits a modulo $3$-orientation $D_j$ and $G-e_{j+1}$ admits a modulo $3$-orientation $D_{j+1}$, respectively. Without loss of generality, assume $v_{j-1}\rightarrow v_j$  in $D_j$. Then we have $v_j\rightarrow v_{j+1}$ and $u_{j+1}\rightarrow v_{j+1}$ in $D_j$.  Similarly, WLOG, assume $v_{j+2}\rightarrow v_{j+1}$ in $D_{j+1}$. Then we get $v_{j+1}\rightarrow v_{j}$ and $u_{j}\rightarrow v_{j}$ in $D_{j+1}$.   Besides, the direction of $e$ in $D_{j+1}$ is the same as that in $D_{j}$ for each $e\in E(G) \setminus \{e_j,e_{j+1},v_jv_{j+1}\}$. Thus we have $d^+_{D_{j}}(u_j)=d^+_{D_{j+1}}(u_j)-1$ and $d^-_{D_{j}}(u_j)=d^-_{D_{j+1}}(u_j)$, which implies that 
$u_j$ is not  balanced in $D_{j+1}$ since it is balanced in $D_{j}$,  a contradiction again.
\end{proof}

Kochol \cite{Kochol01,Kochol02} obtained two equivalent statements of Tutte's $3$-flow conjecture as follows: (i) every $5$-edge-connected graph admits a nowhere-zero $3$-flow,  (ii) every bridgeless graph with at most three edge cuts of size three admits a nowhere-zero $3$-flow.  By Theorem \ref{THM: prop3f}, the results of Kochol \cite{Kochol01,Kochol02} can be restates as certain properties of $3$-flow-critical graphs.
\begin{theorem}\label{Kocholthm} {\em (Kochol \cite{Kochol01,Kochol02})}
  Tutte's $3$-flow conjecture is equivalent to each of the following statements.

  (a) Every $3$-flow-critical graph contains a vertex of degree $3$.

  (b) Every $3$-flow-critical graph contains a vertex of degree at most $4$.

  (c) $|V_3(G)|\ge 4$ for every $3$-flow-critical graph $G$.
\end{theorem}
It is proved in \cite{{HanLL16}} that every $Z_3$-reduced graph has a vertex of degree at most $5$, and so, combining Theorem \ref{THM: prop3f}(iii), it implies that every $3$-flow-critical graph contains a vertex of degree at most $5$. 

Theorem \ref{Kocholthm}  may suggest that some better structure properties of $3$-flow-critical graphs could bring new ideas in solving Tutte's $3$-flow conjecture. In particular, Theorem \ref{Kocholthm} shows that Conjecture \ref{CONJ: 2.5n} implies Tutte's $3$-flow conjecture.

Next, we show in details that $K_{3, n-3}^+$ is a $3$-flow-critical graph and that Conjecture \ref{CONJ: 3n} holds for any $3$-flow-critical graph $G$ with $|V_3(G)|\ge |V(G)|-3$ and $|V(G)|\geq9$.
\begin{proposition}\label{k3c}
 For any $n\geq6$, the graph $K_{3, n-3}^+$ is a $3$-flow-critical graph with $3n-8$ edges.
\end{proposition}
\begin{proof}
It is easy to check that  $K_{3, n-3}^+$ has $3n-8$ edges. So it remains to show that $K_{3, n-3}^+$ is $3$-flow-critical. We use notation in Figure 1 to label the vertices of $K_{3, n-3}^+$, and let $X=\{x_1,x_2,x_3\}$ and $Y=\{y_1,y_2,\ldots,y_{n-3}\}$. To the contrary, suppose $K_{3, n-3}^+$ admits a modulo $3$-orientation $D$. Since  all vertices in $Y$ are 3-vertices,  we have $d^+_D(y_i)=3$ or $d^-_D(y_i)=3$ for each $y_i\in Y$. And it is easy to check that $d^+_{D}(x_1) - d^-_{D}(x_1)\not\equiv 0\pmod3$ if $d^+_{D}(x_3) - d^-_{D}(x_3)\equiv 0\pmod3$, since $x_1$ has an extra neighbor $x_2$.  Hence $K_{3, n-3}^+$  does not admit a modulo $3$-orientation. For any $e\in E(K_{3, n-3}^+)$, in order to show that $G'=K_{3, n-3}^+/e$ has a modulo $3$-orientation, it is sufficient to prove that $G''=K_{3, n-3}^+-e$ has a modulo $3$-orientation.

We firstly give a special orientation of the complete bipartite graph $K_{3, t-3}$ with $t\geq5$. Let $X=\{x_1,x_2,x_3\}$ and $Y=\{y_1,y_2,\ldots,y_{t-3}\}$ be the two parts of $K_{3, t-3}$. Assign to each edge incident to $x_1$ a direction such that $d^+(x_1)-d^-(x_1) \equiv k \pmod 3$. And assign directions to the remain edges such that $d^+(v)-d^-(v) \equiv 0 \pmod 3$ for each $v\in Y$. Then we obtain an orientation $D(k)$ of $K_{3, t-3}$
such that $d_{D(k)}^+(u)-d_{D(k)}^-(u) \equiv k \pmod 3$ for each $u\in X$, and $d_{D(k)}^+(v)-d_{D(k)}^-(v) \equiv 0 \pmod 3$ for each $v\in Y$.

Now by symmetry, we consider three cases $e=x_1x_2$, $e=x_1y_1$ and $e=x_3y_1$.
If $e=x_1x_2$, then $G''\cong K_{3, n-3}$. So $G''$ has a modulo $3$-orientation $D(k)$ with $k=0$.
If $e=x_1y_1$, then $G_1=G''-y_1-\{x_1x_2\}$ is isomorphic to $K_{3, n-4}$. So $G_1$ has an orientation $D(k)$ with $k=1$. With the restriction of $D(1)$ on $G''$, we obtain a  modulo $3$-orientation of $G''$ by assigning $x_2\rightarrow x_1$, $x_2\rightarrow y_1$ and $y_1\rightarrow x_3$. If $e=x_3y_1$, then $G_1=G''-y_1-\{x_1x_2\}$ is isomorphic to $K_{3, n-4}$. So $G_1$ has an orientation $D(k)$ with $k=0$. With the restriction of $D(0)$ on $G''$, we obtain a  modulo $3$-orientation of $G''$ by assigning $x_1\rightarrow x_2$, $x_2\rightarrow y_1$ and $y_1\rightarrow x_1$.

Thus, for all cases above, we can obtain a modulo $3$-orientation of $G''$. Hence we conclude that $K_{3, n-3}^+$ is $3$-flow-critical.
\end{proof}

\begin{proposition}\label{up3n-8}
  Let $G$ be a $3$-flow-critical graph on $n\ge 9$ vertices. If $|V_3(G)|\ge n-3$, then $$|E(G)|\le 3n-8.$$
  Moreover, the equality holds if and only if $G\cong K_{3, n-3}^+$.
\end{proposition}
\begin{proof}
By Lemma \ref{lem1} and Theorem \ref{THM: prop3f}(iii), $G$ contains no parallel edges. We consider three cases in the following. Firstly, suppose $|V_3|\ge n-1$.  By Theorem \ref{THM: prop3f}(iv), the graph $G$ is an odd wheel and $|E(G)|\leq 2n-2$, which is less than $3n-8$ when $n\geq 9$. Then suppose $|V_3|= n-2$. Assume $G[V_3]$ has $t$ components. By Theorem \ref{THM: prop3f}(iv), we know $G[V_3]$ is a forest, and hence $|E(G)|=|E(G[V_3])|+|[V_3,V_{\geq4}]|+|E(G[V_{\geq4}])|\leq (n-2-t)+(3(n-2)-2(n-2-t))+1=2n+t-3$. Since $G$ has no parallel edges and $G[V_3]$ has no isolated vertex, we obtain
$t\leq \lfloor \frac{n-2}{2} \rfloor$, which implies $|E(G)|<3n-8$ by $n\geq9$.

Finally, suppose $|V_3|=n-3$. Let $i=|E(G[V_{\geq4}])|$ and $V_{\geq4}=\{u_1,u_2,u_3\}$. Then $t\leq n-3$ and $0\leq i\leq3$. So we have $|E(G)|\leq (n-3-t)+(3(n-3)-2(n-3-t))+i=2n+t+i-6$. If $t+i\leq n-3$, then $|E(G)|\leq 3n-9$. Now we consider the case $t+i\geq n-2$, whereas $i\geq1$.  If $i=1$, then $t=n-3$ and  $G=K_{3, n-3}^+$. If $2\leq i\leq 3$, then $t\geq n-5$ and we assume $\{u_1u_2,u_2u_3\}\subseteq E(G[V_{\geq4}])$ by symmetry. Since $n\geq 9$, there are at least two vertices, denoted by $v_1$ and $v_2$,  are isolated in $G[V_3]$.
We use $H$ to denote the graph induced by $\{v_1,v_2,u_1,u_2,u_3\}$. Let $H'=H$ if $u_1u_3\notin E(G)$ and $H'=H-u_1u_3$ if $u_1u_3\in E(G)$. So $H'$ is a wheel $W_4$ and is $Z_3$-connected by Lemma \ref{lem-wheel}, which contradicts Theorem \ref{THM: prop3f}(iii).
Hence $K_{3, n-3}^+$ is the only extreme graph to attain the bound.
\end{proof}

Now we apply Theorem \ref{THM: prop3f} and a counting argument to obtain the lower bound in Theorem \ref{THM: main4n}. Since for an odd wheel $W_{n-1}$ we have $|E(W_{n-1})|=2n-2\geq\frac{8n+2}{5}$ if $n\geq5$, it suffices to prove the following theorem.
\begin{theorem}\label{THM: mainlower}
  For any $3$-flow-critical graph $G$ on $n$ vertices other than an odd wheel,
  $$|E(G)|\ge \frac{8n+2}{5}.$$
\end{theorem}
\begin{proof}
  We count the number of edges in $[V_3, V_3^c]$ with two aspects to build an inequality.

  On one hand, since $G[V_3]$ contains no cycle by Theorem \ref{THM: prop3f}(iv), we have
  \begin{eqnarray*}
    |[V_3, V_3^c]|=3|V_3|-2|E(G[V_3])|
    \ge 3|V_3|-2(|V_3|-1)=|V_3|+2.
  \end{eqnarray*}

  On the other hand, it follows from $V_3^c=V_{\ge 4}$ that
  \begin{eqnarray*}
    |[V_3, V_3^c]|&=&\sum_{k\ge 4}k|V_k|-2|E(G[V_{\ge 4}])|
    \le\sum_{k\ge 3}k|V_k| - 3|V_3|=2|E(G)|-3|V_3|.
  \end{eqnarray*}
  Hence we have $2|E(G)|-3|V_3|\ge |V_3|+2$, and so $2|E(G)|\ge 4|V_3|+2$. Therefore, we have
  \begin{eqnarray*}
    |E(G)|&=&\frac{8|E(G)|}{10}+\frac{2|E(G)|}{10}\\
    &\ge&\frac{4}{10}\sum_{k\ge 3}k|V_k|+\frac{1}{10}(4|V_3|+2)\\
    &\ge&\frac{8}{5}\sum_{k\ge 3}|V_k| +\frac{1}{5}\\
    &=& \frac{8n+1}{5}.
  \end{eqnarray*}
  To obtain the bound $\frac{8n+2}{5}$ in the theorem, we shall show that $|E(G)|\neq \frac{8n+1}{5}$ below.

 Suppose to the contrary that $|E(G)|= \frac{8n+1}{5}$. Then all the inequalities above are equalities. In particular, we have that $G[V_3]$ is a tree and $V_{\ge 4}=V_4$ is an independent set. Let $x_1$ be a leaf vertex of the tree $G[V_3]$, and let $y$ be a neighbor of $x_1$ with degree $4$. Suppose the neighbors of $y$ are $x_1,x_2,x_3,x_4$, where $x_i\in V_3$ for each $i\in\{1,2,3,4\}$. Since $G[V_3]$ is a tree, there is a unique path, say $P_{ij}$, connecting the vertices $x_i$ and $x_j$ in $G[V_3]$.
Then by symmetry, we consider two cases as follows.

 {\em Case 1.}\ $x_2\in V(P_{13})$ but $x_4\notin  V(P_{13})$.

Let $G'=G-yx_4$. Since $G$ is 3-flow-critical, by Theorem2.1(i),  we have that  $G'$ admits a  modulo $3$-orientation $D'$. This implies that $d_{D'}^+(y)=3$ or $d_{D'}^-(y)=3$.  Thus $|V(P_{13})|$ is odd since all vertices in the path $P_{13}$ belong to $V_3$. Let $G''=G-yx_2$. By Theorem \ref{THM: prop3f}(i), $G''$ has a modulo $3$-orientation $D''$ and has $d_{D''}^+(y)=3$ or $d_{D''}^-(y)=3$. However, the edges $yx_1$ and $yx_3$ must have opposite directions in $D''$ since $|V(P_{13})|$ is odd and $d_{G''}(x_2)=2$, i.e., $y\rightarrow x_1$ if $x_3\rightarrow y$ and $y\rightarrow x_3$ if $x_1\rightarrow y$. This is a contradiction.

 {\em Case 2.}\ $x_i\notin V(P_{jk})$ for any $\{i,j,k\} \subseteq \{1,2,3,4\}$.

By  {Case 1}, we know $|V(P_{ij})|$ is an odd number for each index pair $\{i,j\} \subseteq \{1,2,3,4\}$. Since $x_1$ is a leaf of the tree $G[V_3]$, there is a neighbor  $z$ of $x_1$ such that $z\neq y$ and $z\in V_4$. Let $G'=G-zx_1$.  Since $G$ is 3-flow-critical, we have that  $G'$ admits a  modulo $3$-orientation $D'$. Since $|V(P_{23})|$ and $|V(P_{34})|$ are both odd and $V(P_{23})\cup V(P_{34})\subseteq V_3$, we have that the
edges $yx_2$, $yx_3$ and $yx_4$  are all leaving or all entering  $y$ in $D'$. It implies that $d^+_{D'}(y)\geq3$ or $d^-_{D'}(y)\geq3$. Then we know $d^+_{D'}(y) - d^-_{D'}(y)\not\equiv 0\pmod3$ since $d_{G'}(y)=4$, a contradiction.
\end{proof}

\section{Upper Bounds and $Z_3$-reduced Graphs}
In this section, we develop a  method to prove an  upper bound for $3$-flow-critical graphs, which is tight for $K_4$.
The following definition is motivated by the theorem of Nash-Williams \cite{Nash61} and Tutte \cite{Tutte61} on spanning tree packing.
\begin{definition}
  Let ${\mathcal X}=\{X_1, X_2,\dots, X_t\}$ be a partition of $V(G)$. Define $$\rho_G({\mathcal X})=\sum_{i=1}^{t}d_G(X_i)-8t+20$$
  and
  $$\rho(G)=\min\{\rho_G({\mathcal X}): {\mathcal X}\ is\ a\ partition\ of\ V(G)\},$$
  where  $d_G(X_i)=|[X_i,X_i^c]_G|$.
\end{definition}

For a graph $G$ with few vertices, it is easy to determine $\rho(G)$. For example, $\rho(K_2)=6, \rho(K_3)=2, \rho(P_3)=0$, and $\rho(K_4)=0$. Note that for these  graphs, $\rho(G)$ is attained only by the trivial partition, which is a partition with exact one vertex in each part.

  For a partition ${\mathcal X}=\{X_1, X_2,\dots, X_t\}$  of $V(G)$, let $G/\mathcal{X}$ be the graph obtained by identifying all vertices in each $X_i$ to form a new vertex $x_i$.  We say a graph $G$ is {\em $Z_3$-reduced} to a graph $H$ if $H$ is obtained from $G$ by contracting all its $Z_3$-connected subgraphs consecutively. In other words, there exists a partition ${\mathcal X}=\{X_1, X_2,\dots, X_t\}$  of $V(G)$ such that  $G/{\mathcal X}=H$ and $G[X_i]$ is $Z_3$-connected for each $i\le t$ (possibly $G[X_i]=K_1$).

  The main result of this section is  the following theorem.

\begin{theorem}\label{THM: main}
  Let $G$ be a  connected graph with $\rho(G)\ge 0$. Then either\\
  (i) $G$ is $Z_3$-connected, or\\
  (ii) $G$ can be $Z_3$-reduced to one of the graphs $K_2, K_3, P_3, K_4$.
\end{theorem}

\begin{proof}
Assume, by way of contradiction, the result is false and study a minimal counterexample $G$ with respected to $|V(G)|+|E(G)|$. That is, $G$ is not $Z_3$-connected and $G$ cannot be $Z_3$-reduced to one of the graphs $K_2, K_3, P_3, K_4$. We first present some preliminary reductions on $G$.

  \begin{claim}\label{CL: noZ3subgraph}
    $G$ is  $Z_3$-reduced and $|V(G)|\ge 7$.
  \end{claim}

 \begin{proof}
 Suppose to the contrary that there exists a subgraph $H$ of $G$ such that $H$ is $Z_3$-connected, where $|V(H)|>1$. Since $G$ is a minimal counterexample,  we consider two cases as follows. If $G/H$ is $Z_3$-connected, then by Lemma \ref{lem1}, $G$ is $Z_3$-connected, a contradiction. If $G/H$ can be $Z_3$-reduced to one of the graphs $K_2, K_3, P_3, K_4$, then by definition $G$ is $Z_3$-reduced to one of the graphs $K_2, K_3, P_3, K_4$. Each case leads to a contradiction. Hence $G$ is $Z_3$-reduced and contains no nontrivial $Z_3$-connected subgraph.

  Clearly, we have $|V(G)|\ge 3$. It is routine to verify that $|V(G)|\ge 7$ by some case analysis, but we shall apply a basic fact in \cite{LLW18} to accomplish this work. By Lemma 2.10 in \cite{LLW18}, when $n=3,4,5,6$, any $Z_3$-reduced graph on $n$ vertices contain at most $3,6,8,11$ edges, respectively. As $\rho(G)\ge 0$, $G$ contains at least $2,6,10,14$ edges when $|V(G)|=3,4,5,6$, respectively. Thus either $G\in \{K_3, P_3, K_4\}$ or $G$ is not $Z_3$-reduced, a contradiction. This shows $|V(G)|\ge 7$.
 \end{proof}

For a partition $\mathcal{X}=\{X_1,X_2,\ldots, X_t\}$ of $V(G)$, if $G[X_i]=K_a$ is a complete graph for some $i$ and $|X_k|=1$ for each $k\neq i$, where $a$ is a positive integer, then we call $\mathcal{X}$ a {\em $K_a$-trivial partition}. A $K_1$-trivial partition is simply a trivial partition.

\begin{claim}\label{rhoH}  Let $H$ be a proper subgraph of $G$ with $|V(H)|>1$. Assume that $\rho_H({\mathcal Y})\ge 7$ for any nontrivial partition ${\cal Y}$ of $H$. Let ${\mathcal Z}$ denote the trivial partition of $H$. Then each of the following holds.\\
(i)  The trivial partition ${\mathcal Z}$ of $H$ satisfies $\rho_H({\mathcal Z})\le 6$.\\
(ii) If $\rho_H({\mathcal Z})\ge 3$, then $H\cong K_2$.\\
(iii) If $\rho_H({\mathcal Z})\ge 1$, then $H\in \{K_2,K_3\}$.
  \end{claim}

\begin{proof} Since $G$ is a minimal counterexample to Theorem \ref{THM: main}, the theorem is applied for its proper subgraph $H$.  Assume that the trivial partition ${\mathcal Z}$ of $H$ satisfies $\rho_H({\mathcal Z})\ge 0$. Then Theorem \ref{THM: main} implies that either $H$ is $Z_3$-connected, or $H$ can be $Z_3$-reduced to one of the graphs $K_2, K_3, P_3, K_4$. As $G$ is $Z_3$-reduced, $H$ and any subgraph of $H$ are not $Z_3$-connected.  Note that $H$ contains no nontrivial $Z_3$-connected subgraph, and so the $Z_3$-reduction of $H$ is itself. So Theorem \ref{THM: main} implies that $H\in \{K_2, K_3, P_3, K_4\}$.

Notice also that for any partition ${\cal Y}$ of $H$, we have $\rho(H/{\mathcal Y})\ge \rho(H)$, since any  partition of $H/{\mathcal Y}$ can be obtained from a partition of $H$ by collapsing vertex sets in ${\mathcal Y}$ to become vertices.  Recall  $\rho(K_2)=6, \rho(K_3)=2, \rho(P_3)=0$, and $\rho(K_4)=0$. Hence for any nontrivial partition ${\cal Y}$ of $H$, $\rho(H/{\cal Y})\ge 7$ implies that $H/{\cal Y}\notin \{K_2, K_3, P_3, K_4\}$.

(i) Suppose to the contrary that $\rho_H({\mathcal Z})\ge 7$ for the trivial partition ${\cal Z}$ of $H$. Then  for any partition ${\cal Y}$ of $H$, $\rho(H/{\cal Y})\ge 7$ and $H/{\cal Y}\notin \{K_2, K_3, P_3, K_4\}$. This contradicts to $H\in \{K_2, K_3, P_3, K_4\}$.

(ii) We have that $\rho_H({\mathcal Z})\ge 3$ implies $H\notin \{K_3, P_3, K_4\}$, and so $H\cong K_2$.

(iii) We have that  $\rho_H({\mathcal Z})\ge 1$ implies $H\notin \{P_3, K_4\}$, and so  $H\in  \{K_2,K_3\}$.
\end{proof}

  \begin{claim}\label{rhop}
    Let ${\cal X}$ be a nontrivial partition of $G$. Then we have\\
    (i) $\rho_G({\cal X})\ge 6$;\\
    (ii) $\rho_G({\cal X})\ge 10$ if ${\cal X}$ is not a $K_2$-trivial partition;\\
    (iii) $\rho_G({\cal X})\ge 12$ if ${\cal X}$ is neither a $K_2$-trivial partition nor a $K_3$-trivial partition.
  \end{claim}

\begin{proof} Let $\mathcal{X}=\{X_1,X_2,\ldots, X_t\}$. If $t=1$, then it is easy to verify $\rho_G({\cal X})=12$. So we assume $t\geq2$ and  $|X_1|> 1$. Let $H=G[X_1]$.
For any partition ${\mathcal Y}=\{Y_1, Y_2,\dots, Y_s\}$ of $V(H)$,  ${\mathcal Y}\cup ({\mathcal X}\setminus\{X_1\})$ is a partition of $V(G)$, and we have
  \begin{eqnarray}\nonumber
&~& \rho_G({\mathcal Y}\cup ({\mathcal X}\setminus\{X_1\}))\\\nonumber
 &=&\sum_{j=1}^sd_G(Y_j) + \sum_{i=2}^td_G(X_i) -8(s+t-1)+20\\\nonumber
 &=&[\sum_{j=1}^sd_G(Y_j)-d_G(X_1)-8s+20] + [\sum_{i=1}^td_G(X_i) -8(t-1)]\\\nonumber
 &=&[\sum_{j=1}^sd_H(Y_j)-8s+20] + [\sum_{i=1}^td_G(X_j) -8t+20]-12\\\nonumber
 &=& \rho_H({\mathcal Y})+ \rho_G({\mathcal X})-12.
\end{eqnarray}

That is,
\begin{eqnarray}\label{EQ: rPrQ}
 \rho_H({\mathcal Y})= \rho_G({\mathcal Y}\cup ({\mathcal X}\setminus\{X_1\})) - \rho_G({\mathcal X})+12.
\end{eqnarray}

(i) Suppose to the contrary that $\rho_G({\mathcal X})\le 5$. Then for any partition ${\cal Y}$ of $H$, we have $\rho_H({\mathcal Y})\ge 7$  by Eq.(\ref{EQ: rPrQ}), contradicting to Claim \ref{rhoH}(i).

(ii) We first show that $\rho_G({\cal X})\ge 12$ if ${\cal X}$ is a partition with $|X_1|>1$ and $|X_2|>1$.
Suppose that $\rho_G({\mathcal X})\le 11$. Since $|X_2|>1$,  the partition ${\mathcal Y}\cup ({\mathcal X}\setminus\{X_1\})$ is  a nontrivial partition of $G$. So $\rho_G({\mathcal Y}\cup ({\mathcal X}\setminus\{X_1\}))\geq6$ by (i).
Then we have $\rho_H({\mathcal Y})\ge 7$  for any partition ${\cal Y}$ of $H$ by Eq.(\ref{EQ: rPrQ}), contradicting to Claim \ref{rhoH}(i).


Now suppose that $\rho_G({\mathcal X})\le 9$, $|X_1|\ge 2$ and $|X_i|= 1$ for each $i\in \{2,3,\ldots,t\}$.  By Eq.(\ref{EQ: rPrQ}) and (i), we have $\rho_H({\mathcal Y})\ge  0-9+12=3$  for any  partition ${\cal Y}$ of $H$, and furthermore, $\rho_H({\mathcal Y}_1)\ge 6-9+12=9$ for any  nontrivial partition ${\cal Y}_1$ of $H$.  Thus $H\cong K_2$ by  Claim \ref{rhoH}(ii).


(iii) Suppose that $\rho_G({\mathcal X})\le 11$, $|X_1|\ge 2$ and $|X_i|= 1$ for each $i\in \{2,3,\ldots,t\}$.  By Eq.(\ref{EQ: rPrQ}) and $(i)$, we have $\rho_H({\mathcal Y})\ge 0-11+12= 1$ for any  partition ${\cal Y}$ of $H$, and moreover, $\rho_H({\mathcal Y}_1)\ge 6-11+12= 7$  for any  nontrivial partition ${\cal Y}_1$ of $H$.  Thus  $H\in \{K_2, K_3\}$ by  Claim \ref{rhoH}(iii).
%
\end{proof}

\begin{claim}\label{CL: ess7}
 For any nonempty vertex subset $S\subsetneq V(G)$,\\
 (i) $d(S)\ge 4$. That is,  $G$ is $4$-edge-connected.\\
 (ii) If $|S|\ge 2$ and $|S^c|\ge 3$, then $d(S)\ge 7$. That is,  $G$ is essentially  $7$-edge-connected.
\end{claim}

\begin{proof}
It is obvious that  $\{S,S^c\}$ is a partition of $V(G)$.

 (i) Since $|V(G)|\geq 7$,  by Claim \ref{rhop}, we have that $12\leq \rho_G(\{S,S^c\})=2|[S,S^c]|-16+20$, which yields $|[S,S^c]|\geq4$. This implies that $G$ is $4$-edge-connected.

(ii) It is sufficient to prove that if $|S|\ge 2$ and $|S^c|\ge 3$, then  $\rho_G(\{S,S^c\} )\geq18$. It is clear that if  $\rho_G(\{S,S^c\} )\geq18$, then  we have $|[S,S^c]|\geq7$ by $\rho_G(\{S,S^c\})=2|[S,S^c]|-16+20$. Now let us prove $ \rho_G(\{S,S^c\} )\geq18$. By contradiction, suppose $\rho_G(\{S,S^c\} )\leq 17$.
Since $|V(G)|\geq 7$, by symmetry, we assume $|S^c|\ge 4$. Let $H=G[S]$. By Eq.(\ref{EQ: rPrQ}) and Claim \ref{rhop}, we have  $\rho_H({\mathcal Y})\ge 12-17+12=7$  for any partition ${\cal Y}$ of $H$, a contradiction to Claim \ref{rhoH}(i).
\end{proof}

Next we introduce a few more tools in order to complete the proof of Theorem \ref{THM: main}.
%
Some splitting operation could keep the resulting graph $Z_3$-connected as follows.
\begin{lemma}(Lemma 4.1 of \cite{HanLL16})\label{lem2}
Let $G$ be a graph, $z$ be a vertex of $G$ with degree at least $4$ and ${zv_1,zv_2}\in E_G(z)$. If $G'=G-z+v_1v_2$ is $Z_3$-connected, then $G$ is $Z_3$-connected.
\end{lemma}
Another key result is the following theorem due to Lov\'asz, Thomassen, Wu and Zhang \cite{LTWZ13}.
\begin{theorem}\label{thmLTWZ} {\em (Lov\'asz et al. \cite{LTWZ13})}
  Every $6$-edge-connected graph is $Z_3$-connected.
\end{theorem}

%

Now we are ready to finish the proof. By Claim \ref{CL: ess7}(ii), each nontrivial edge cut of $G$ has size at least $7$. But $G$ is not $6$-edge-connected by Theorem \ref{thmLTWZ}. Hence the minimal degree of $G$ is at most $5$. Let $z$ be a minimal degree vertex of $G$. Then   by Claim \ref{CL: ess7}(i) we have $$4\leq d_G(z)\leq5.$$
Our main strategy below is to show that by  Claim \ref{CL: ess7} it is always possible to  select ${zv_1,zv_2}\in E_G(z)$ such that the modified graph $G'=G-z+v_1v_2$ still satisfies the condition of Theorem \ref{THM: main}. Then the minimality of $G$ and Theorem \ref{THM: main} would imply that $G'$ is $Z_3$-connected. Hence, $G$ is  $Z_3$-connected by Lemma \ref{lem2}, a contradiction to Claim \ref{CL: noZ3subgraph}.

\begin{claim}\label{zdegree4}
  If $d_G(z)=4$, then $G$ is $Z_3$-connected.
\end{claim}
\begin{proof}
Since $G$ is $Z_3$-reduced by Claim \ref{CL: noZ3subgraph}, $G$ contains no parallel edges. So $z$ has four neighbors, and we may let $N_G(z)=\{v_1,v_2,v_3,v_4\}$. Denote $G'=G-z+v_1v_2$.

We first claim that $G'$ is still $4$-edge-connected. Since $G$ is essential $7$-edge-connected by Claim \ref{CL: ess7}(ii), $z$ is not adjacent to any $4$-vertex. So $d_G(v_i)\ge 5$ for each $i\in\{1,2,3,4\}$. Hence $G'=G-z+v_1v_2$ has minimal degree at least $4$. Suppose to the contrary that there is a subset $S\subset V(G')$ such that $|[S,S^c]_{G'}|\leq 3$. Note that $|S|\geq2$ and $|S^c|\geq2$. We may, WLOG, assume that $v_1\in S$. Let $T=S\cup\{z\}$. Then $|[T,T^c]_G|\le |[S,S^c]_{G'}|+(d_G(z)-1)\leq 6$ since $v_1z\notin [T,T^c]_G$.  This  contradicts to Claim \ref{CL: ess7}(ii) that $G$ is essential $7$-edge-connected. Hence, $G'$ is still $4$-edge-connected.

Next, we claim that $\rho(G')\ge 0$.  For any  partition ${\cal X'}=\{X_1',X_2',\ldots, X_t'\}$ of $G'$, we can obtain a partition  ${\cal X}=\{z\}\cup {\cal X'}$ of $G$. There are at most four edges in $E_G(z)$ that is counted in $\rho_G({\cal X})$ but no in $\rho_{G'}({\cal X'})$. Hence we have $0\leq\rho_G({\mathcal X})=\sum_{i=1}^{t}d_G(X'_i)+d_G(z)-8(t+1)+20\leq (\sum_{i=1}^{t}d_{G'}(X'_i)+4)+4-8t-8+20=\rho_{G'}({\mathcal X'})$. Thus $\rho_{G'}({\mathcal X'})\ge 0$ for any partition ${\cal X'}$ of $G'$.

Now the minimality of $G$ implies that Theorem \ref{THM: main} is applied for $G'$. Thus either $G'$ is $Z_3$-connected, or there is a partition ${\cal Y}$ of $G'$ such that $G'/{\cal Y}\in \{K_2, K_3, P_3, K_4\}$.  But the latter case cannot happen since $G'$ is  $4$-edge-connected. Hence $G'$ is $Z_3$-connected, and so $G$ is $Z_3$-connected by Lemma \ref{lem2}.
\end{proof}

By Claims \ref{CL: noZ3subgraph}, \ref{CL: ess7}, \ref{zdegree4}, we conclude that $G$ is $5$-edge-connected. Then we present the last claim to get a contradiction.
\begin{claim}\label{zdegree5}
  If $d(z)=5$, then $G$ is $Z_3$-connected.
\end{claim}
\begin{proof}
 Similar as the proof of Claim \ref{zdegree4}, $G$ contains no parallel edges and we may let $v_1,v_2\in N_G(z)$ be distinct neighbors of $z$. Let $G'=G-z+v_1v_2$.
Then the minimum degree of $G'=G-z+v_1v_2$ is at least $4$ since $G$ is $5$-edge-connected.

Now we prove that $G'$ is  $4$-edge-connected. Suppose, by contradiction, that there is a subset $S\subset V(G')$ such that $|[S,S^c]_{G'}|\leq 3$  in $G'$. Note that $|S|\geq2$ and $|S^c|\geq2$. As $|N_G(z)|=5$, we may, by symmetry, assume that $S$ contains at least $2$ neighbors of $z$ in $G$. Let $T=S\cup\{z\}$. Then $[T,T^c]_G$ contains at most $3$ edges in $E_G(z)$ that may not in $[S,S^c]_{G'}$, and so $|[T,T^c]_G|\le |[S,S^c]_{G'}|+3\leq 6$,  which contradicts to Claim \ref{CL: ess7}(ii).

Then we show that $\rho_{G'}({\mathcal X'})\ge 0$ for any partition ${\cal X'}$ of $G'$.
  For any  partition ${\cal X'}=\{X_1',X_2',\ldots, X_t'\}$ of $G'$, we can obtain a partition  ${\cal X}=\{z\}\cup {\cal X'}$ of $G$.
  For the trivial partition ${\cal Z'}$ of $G'$ and the trivial partition ${\cal Z}$ of $G$, each edge in $E(G)\cup\{v_1v_2\}\setminus E_G(z)$ is counted in $\rho_{G'}({\cal Z'})$, and so we have $\rho_{G'}({\cal Z'})=\rho_{G}({\cal Z})+2+8-2d_G(z)\ge 0$. For any nontrivial partition ${\cal X'}$ of $G'$, ${\cal X}=\{z\}\cup {\cal X'}$ is a nontrivial partition of $G$, and by Claim \ref{rhop}, we have $6\leq\rho_G({\mathcal X})=\sum_{i=1}^{t}d_G(X'_i)+d_G(z)-8(t+1)+20\leq (\sum_{i=1}^{t}d_{G'}(X'_i)+5)+5-8t-8+20=\rho_{G'}({\mathcal X'})+2$, which implies $\rho_{G'}({\cal X'})\ge 4$. Hence we conclude that $\rho(G')\ge 0$.

Since $G$ is a minimal counterexample, applying Theorem \ref{THM: main} to $G'$, we have that either $G'$ is $Z_3$-connected, or there is a partition ${\cal Y}$ of $G'$ such that $G'/{\cal Y}\in \{K_2, K_3, P_3, K_4\}$. But we know $G'/{\cal X'}\notin \{K_2, K_3, P_3, K_4\}$ for any partition ${\cal X'}$ of $G'$, since $G'$ is $4$-edge-connected. Thus $G'$ is $Z_3$-connected, and hence $G$ is $Z_3$-connected by Lemma \ref{lem2}.
%
\end{proof}

With Claims \ref{CL: noZ3subgraph}, \ref{zdegree4}, and \ref{zdegree5}, we obtain a contradiction.
This completes the proof of Theorem \ref{THM: main}.
\end{proof}

If a graph  contains $4$ edge-disjoint spanning trees, then $\rho(G)\ge 12$ and $G/{\cal X}\notin \{K_2, K_3,$ $P_3, K_4\}$ for any partition ${\cal X}$ of $G$, and so $G$ is $Z_3$-connected by Theorem \ref{THM: main}. This reproves the main result in \cite{HanLL16} below. Actually, Theorem \ref{THM: main} is an improvement of the result in \cite{HanLL16}.
\begin{corollary}\label{coro4trees} \cite{HanLL16}
  Every graph with four edge-disjoint spanning trees is $Z_3$-connected.
\end{corollary}

We also need the following Nash-Williams Arboricity Theorem \cite{Nash64}.
\begin{theorem}\label{thm-nsthm} {\em (Nash-Williams Arboricity Theorem \cite{Nash64})}
  A graph $G$ decomposes  into $k$ forests if and only if every  subgraph $H$ of $G$ satisfies $|E(H)|\le k(|V(H)|-1)$.
\end{theorem}

{\bf Proof of Upper bound of Theorem \ref{THM: main4n} using Theorem \ref{THM: main}:} Let $G$ be a $3$-flow-critical graph on $n\ge 8$ vertices. By Theorem \ref{THM: prop3f} (iii), $G$ is $Z_3$-reduced. By Corollary \ref{coro4trees}, $G$ and each nontrivial subgraph of $G$ contains no $4$ edge-disjoint spanning trees. We further claim that every subgraph $H$ of $G$ satisfies $|E(H)|\le 4(|V(H)|-1)$ below. Suppose not, and let $H$ be a subgraph of $G$ satisfying $|E(H)|> 4(|V(H)|-1)$ with $|V(H)|$ minimized. Clearly, $|V(H)|\ge 4$ since $H$ is $Z_3$-reduced. For the trivial partition ${\cal Z}$ of $H$, $\rho_H({\cal Z})=2|E(H)|-8|V(H)|+20>12$. We show  $\rho(H)\ge 12$ to get a contradiction. Let $\mathcal{X}=\{X_1,X_2,\ldots, X_t\}$ be a partition of $H$ with the size of $\rho_H({\mathcal X})$ as small as possible. By contradiction, suppose  $\rho_H({\mathcal X})<12$. Then $\mathcal{X}$ is a nontrivial partition of $H$, and WLOG, let $|X_1|>1$ and denote $\Gamma=H[X_1]$. By a calculation same to that of  Eq.(\ref{EQ: rPrQ}), for any partition ${\cal Y}$ of $\Gamma$, we have $\rho_{\Gamma}({\mathcal Y})= \rho_H({\mathcal Y}\cup ({\mathcal X}\setminus\{X_1\})) - \rho_H({\mathcal X})+12\ge 12$, since $\rho_H({\mathcal Y}\cup ({\mathcal X}\setminus\{X_1\})) \ge \rho_H({\mathcal X})$ by the minimality of $\rho_H({\mathcal X})$. Hence $\rho(\Gamma)\ge 12$ and the subgraph $\Gamma$ is $Z_3$-connected by Theorem \ref{THM: main}, contradicting that $G$ is $Z_3$-reduced. Thus  $\rho_H({\mathcal X})\ge 12$ for any partition of $H$, and so $H$ is $Z_3$-connected by Theorem \ref{THM: main}, a contradiction again. This contradiction implies the claim that every subgraph $H$ of $G$ satisfies $|E(H)|\le 4(|V(H)|-1)$.
By Nash-Williams Arboricity Theorem (Theorem \ref{thm-nsthm}), $G$ can be decomposable into $4$ edge-disjoint forests, say $F_1,F_2,F_3,F_4$.

Now suppose that $|E(G)|\ge 4n-10$. Then $4n-5\geq|E(F_1)|+|E(F_2)|+|E(F_3)|+|E(F_4)|=|E(G)|\ge 4n-10$.   Let $T'_i$ be a tree obtained by adding some edges to $F_i$ for each $i\in\{1,2,3,4\}$. Let $G'$ be the edge-disjoint union of  trees $T'_1, T'_2, T'_3, T'_4$. Then we have that $|E(G')|=|E(G)|+j$, where $1\leq j\leq 6$. For any  partition ${\cal X}=\{X_1,X_2,\ldots, X_t\}$ of $G$, we have that $12\leq\rho_{G'}({\mathcal X})=\sum_{i=1}^{t}d_{G'}(X_i)-8t+20\leq (\sum_{i=1}^{t}d_G(X_i)+2j)-8t+20=\rho_{G}({\mathcal X})+2j\leq \rho_{G}({\mathcal X})+12$. So $\rho(G)\geq0$.
 Now by Theorem \ref{THM: main}, either  $G$ is $Z_3$-connected, or $G$ can be $Z_3$-reduced to one of the graphs $K_2, K_3, P_3, K_4$. Since $G$ is $3$-flow-critical, then $G$ is not $Z_3$-connected. Since $G$ is $3$-edge-connected,  $G$ cannot be $Z_3$-reduced to any one of the graphs $K_2, K_3,$ or $P_3$.  Then we have $G\cong K_4$ since $G$ is essentially $4$-edge-connected.
Hence,  $G$ has at most $4n-11$ edges unless $G\cong K_4$.
\ENDproof

~

{\bf Proof of Theorem \ref{THM: 9n8}} By way of contradiction, we suppose $|E(G)|\geq \frac{5n}{2}+9n_8$. If $n_8\geq\frac{n}{6}$, then  $|E(G)|\geq \frac{5n}{2}+\frac{9n}{6}=4n$, which contradicts to Theorem \ref{THM: main4n}. So we have $n_8<\frac{n}{6}$. Since $\delta(G)\geq3$, we have $2|E(G)|=\sum_{v\in V(G)}d(v)\geq 3n_8+9(n-n_8)=9n-6n_8>8n$, still a contradiction to Theorem \ref{THM: main4n}.
\ENDproof

\section{Construction of $3$-flow-critical graphs}
Yao and Zhou \cite{YZ17} proved that for each positive integer $k$, there exists a $4$-critical planar graph with $6k+7$ vertices and $14k+12$ edges.  By duality, their theorem shows the following result on $3$-flow-critical planar graphs.

\begin{theorem} (Yao and Zhou \cite{YZ17})\label{YZthm}
  For each positive integer $k$, there exists a $3$-flow-critical planar graph with $8k+7$ vertices and $14k+12$ edges.
  \end{theorem}

\begin{definition}
Let $G_1$ be a graph with edge $e_1=u_1v_1$, and $G_2$ be a graph with edge $e_2=u_2v_2$.
Let $G_1 \oplus_{(e_1,e_2)} G_2$ be a graph which is obtained from the disjoint union of $G_1-e_1$ and $G_2-e_2$ by identifying $u_1$ and $u_2$ to form a vertex $u$, identifying $v_1$ and $v_2$ to form a vertex $v$, and adding a new edge $uv$.
\end{definition}

\begin{lemma}\label{2sum}
If $G_1$ and $G_2$ are both $3$-flow-critical  graphs with $e_1\in E(G_1)$ and $e_2\in E(G_2)$, then $G_1 \oplus_{(e_1,e_2)} G_2$ is a $3$-flow-critical  graph.
\end{lemma}
\begin{proof}
First, we show that $G_1 \oplus_{(e_1,e_2)} G_2$ has no modulo $3$-orientation. To the contrary, we suppose $G_1 \oplus_{(e_1,e_2)} G_2$ has a modulo $3$-orientation $D$ with $v\rightarrow u$. Let $D_i$ be the restriction of $D$ on $G_i$ for each $i\in\{1,2\}$. Denote $d^+_{D_i}(u_i)-d^-_{D_i}(u_i) \equiv a_i\pmod 3$ and $d^+_{D_i}(v_i)-d^-_{D_i}(v_i) \equiv b_i\pmod 3$. Then we have $a_1+a_2+1 \equiv 0 \pmod 3$ since $u$ is balanced in $D$, and $a_i+b_i \equiv 0 \pmod 3$ since every vertex, except perhaps $u_i$ and $v_i$, is balanced in $D_i$.
If $a_1=0$, then $b_1=0$ and $D_1$ is a modulo $3$-orientation of $G_1$, a contradiction.
If $a_1=1$, then $b_1=2$. And we can obtain a modulo $3$-orientation of $G_1$ by reversing the direction of the arc $v_1u_1$ in $D_1$, a contradiction.
If $a_1=2$, then $a_2=0$ and $b_2=0$, and so $D_2$ is a modulo $3$-orientation of $G_2$, a contradiction again.

Then it suffices to show that  $G_1 \oplus_{(e_1,e_2)} G_2-e$ has a modulo $3$-orientation  for each edge $e$ in $G_1 \oplus_{(e_1,e_2)} G_2$. Recall that $G_i-e'$ has a modulo $3$-orientation for each $e'\in E(G_i)$ by Theorem \ref{THM: prop3f} (i).
If $e=uv$, then the union of the modulo $3$-orientations of $G_i-u_iv_i$ is a modulo $3$-orientation of $G_1 \oplus_{(e_1,e_2)} G_2-e$.
If $e\in E(G_1)$ and $e\neq uv$, then the union of the modulo $3$-orientations of $G_1-e$ and $G_2-u_2v_2$ is a modulo $3$-orientation of $G_1 \oplus_{(e_1,e_2)} G_2-e$.  If $e\in E(G_2)$ and $e\neq uv$, then we can also find a modulo $3$-orientation of $G_1 \oplus_{(e_1,e_2)} G_2-e$ by symmetric argument. This proves that $G_1 \oplus_{(e_1,e_2)} G_2$ is a $3$-flow-critical  graph.
\end{proof}

Finally we apply Theorem \ref{YZthm} and Lemma \ref{2sum} to construct $3$-flow-critical graphs with density from $\frac{7}{4}$ up to $3$.

\begin{theorem}
  For any positive integer $N$ and any rational number $r$ with $\frac{7}{4}<r<3$, there exists a $3$-flow-critical graph $G$ on $n\ge N$ vertices with $$rn-\frac{5}{8}\le |E(G)|\le rn+\frac{5}{8}.$$
\end{theorem}

\begin{proof}
Assume $r=\frac{q}{p}$, where $p,q$ are two positive integers. Note that Lemma \ref{2sum} provides a way to construct  $3$-flow-critical planar graphs from smaller graphs.
Now let $s\geq\frac{6(3p-q)}{8q-14p}+N$ and let $G_1$ be a $3$-flow-critical planar graph with $8s+7$ vertices and $14s+12$ edges as described in Theorem \ref{YZthm}.  Let $$a=\frac{1}{3p-q}((8q-14p)s+5q-3p-\frac{5p}{8})$$ and $$b=\frac{1}{3p-q}((8q-14p)s+5q-3p+\frac{5p}{8}).$$ Since $\frac{7}{4}<\frac{q}{p}<3$, we have $3p-q>0$, $8q-14p>0$ and $5q-3p-\frac{5p}{8}>0$.  So $s>N$ and $a>6$. Since $b-a=\frac{5p}{4(3p-q)}=\frac{5}{4(3-\frac{q}{p})}>1$, there exists  a positive integer  $t$ satisfying $a\leq t\leq b$. Let $G_2=K_{3, t-3}^+$ and let $G=G_1 \oplus_{(e_1,e_2)} G_2$, where $e_1\in E(G_1)$ and  $e_2\in E(G_2)$. Then $G$ is $3$-flow-critical by Lemma \ref{2sum}. By the construction of $G$,  the graph $G$ has $8s+7+t-2=8s+t+5$ vertices and $14s+12+3t-8-1=14s+3t+3$ edges. So $|V(G)|>N$. It is routine to verify that $rn-\frac{5}{8}\le |E(G)|\le rn+\frac{5}{8}$ as follows.

By $a\le t\le b$, we have
\begin{eqnarray*}
  rn+\frac{5}{8}-|E(G)|&=& \frac{q}{p}(8s+t+5)+\frac{5}{8} -(14s+3t+3)\\
  &=& \frac{1}{p}((8s+t+5)q+\frac{5p}{8}-(14s+3t+3)p)\\
  &=&\frac{1}{p}(((8q-14p)s+5q-3p+\frac{5p}{8})- (3p-q)t)\\
   &=&\frac{3p-q}{p}(b-t)\ge 0
\end{eqnarray*}
and
\begin{eqnarray*}
  |E(G)|-(rn-\frac{5}{8})&=&  (14s+3t+3)-\frac{q}{p}(8s+t+5)+\frac{5}{8}\\
  &=& \frac{1}{p}((14s+3t+3)p-(8s+t+5)q+\frac{5p}{8})\\
  &=&\frac{1}{p}( (3p-q)t-((8q-14p)s+5q-3p-\frac{5p}{8}))\\
  &=&\frac{3p-q}{p}(t-a)\ge 0.
\end{eqnarray*}

This completes the proof.
 \end{proof}
 
 \bigskip

\noindent {\bf Acknowledgements.}\  
Jiaao Li was partially supported by National Natural Science Foundation of China (No. 11901318) and Natural Science Foundation of Tianjin (No. 19JCQNJC14100).Yulai Ma and Yongtang Shi were partially supported by the National Natural Science Foundation of China (No. 11922112).
 Weifan Wang was partially supported by the National Natural Science Foundation of China (No. 11771402). Yezhou Wu was partially supported by the National Natural Science Foundation of China (No. 11871426).

 \end{document}